\documentclass{amsart}

\usepackage{amsthm, amsfonts, amssymb, amsmath, latexsym, enumerate, times}
\usepackage{amscd}
\usepackage{xypic}  
\usepackage{amssymb}
\usepackage{amsthm}
\usepackage{epsfig}
\usepackage{paralist}
\usepackage{lmodern}
\usepackage[T1]{fontenc}
\usepackage[utf8]{inputenc}
\usepackage{mathrsfs}
\usepackage{array}
\usepackage{comment}
\usepackage{paralist}
\usepackage{color}

\newtheorem{theorem}{Theorem}[section]
\newtheorem{lemma}[theorem]{Lemma}
\newtheorem{claim}[theorem]{Claim}

\newtheorem{proposition}[theorem]{Proposition}

\theoremstyle{definition}

\newtheorem{remark}[theorem]{Remark}

\renewcommand{\O}{{\mathcal O}}

\newcommand{\p}{{\mathbb P}}

\def\geq{\geqslant}
\def\leq{\leqslant}

\begin{document}
 
\title[On the rationality of certain Fano threefolds]{On the rationality of certain Fano threefolds}

\author{Ciro Ciliberto}
\address{Dipartimento di Matematica, Universit\`a di Roma Tor Vergata, Via O. Raimondo 00173 Roma, Italia}
\email{cilibert@mat.uniroma2.it}

 
 
\maketitle

\medskip

\begin{abstract}  In this paper we study the rationality problem for  Fano threefolds $X\subset \p^{p+1}$ of genus $p$,  that are Gorenstein, with at most canonical singularities. The main results are: (1) a trigonal Fano threefold of genus $p$ is rational  as soon as $p\geq 8$ (this result has already been obtained in \cite {PCS}, but we give here an independent  proof); (2) a non--trigonal Fano threefold of genus $p\geq 7$ 
 containing a plane is rational; (3) any Fano threefold of genus $p\geq 17$ is rational; (4) a Fano threefold of genus $p\geq 12$ containing an ordinary line $\ell$ in its smooth locus is rational.\end{abstract}

\section{Introduction} The problem of rationality of varieties with ample anticanonical divisor is a classical one that goes back more than one century ago with the work of Gino Fano. This problem has been  solved only for smooth Fano varieties of dimension 3 (see \cite {IP}). From the modern point of view of the minimal model programme, it is however important to consider also singular varieties. The rationality problem for singular Fano threefolds has been considered in a series of papers by Yuri Prokorov. 

Let $X\subset \p^{p+1}$ be a Fano threefold of \emph{genus} $p$, i.e., a Gorenstein threefold whose hyperplane divisor is anticanonical, whose general hyperplane section is thus a K3 surface and its general curve section is a canonical curve of genus $p$.

In \cite {P1,P2} Prokhorov classifies non--rational Fano threefold with at worst terminal Gorenstein singularities, rank of the Picard group equal to 1 and genus $p\geq 5$.  The outcome of his classification is essentially that such non--rational Fano threefolds have genus $p$ bounded by 9. In the earlier paper  \cite {P0} Prokhorov considered more generally Gorenstein Fano threefold with at worst  canonical  singularities, and he proved that, if such a variety has at least one non--compound Du Val singular point, then it is rational, except for a few cases that he fully described.

In the present paper, inspired by some ideas of Fano (see \cite {Fa}), we consider the rationality problem in the general case of Gorenstein Fano threefold with at most  canonical  singularities, and no assumption on the Picard group. We prove the following results. After \S \ref {sec:prel} in which we collect some preliminaries,  in \S \ref {sec:trig} we  prove that  any  trigonal Fano threefold of genus $p\geq 8$ is rational (see Theorem \ref {thm:trig}). This result can be deduced from the classification of trigonal Gorenstein Fano threefolds with at most  canonical singularities given in \cite {PCS}. Our proof here however does not resort to this classification, is rather fast and conceptually easy and relies only on projective geometric arguments. Then in \S \ref {sec:plane} we prove that any  Gorenstein Fano threefold of genus $p\geq 7$ with at most  canonical singularities, that is non--trigonal and contains a plane is rational (see Theorem \ref {thm:propl}). The rationality question of Fano threfolds of genus $p\leq 6$ containing a plane is very interesting and it is  discussed at the end of \S  \ref {sec:plane}, in which we give various examples. Finally in \S \ref {sec:appl} we give two applications of Theorem \ref {thm:propl}. The first one is the quite general Theorem \ref {thm:ratbig} to the effect that any non--trigonal Fano threefold of genus $p\geq 17$ is rational. The second one is Theorem \ref {thm:line} that says that any Fano threefold of genus $p\geq 12$ containing an ordinary line in its smooth locus is rational (for the definition of \emph{ordinary line} see \S \ref {ssec:lines} below). 

As said, our proofs are inspired by beautiful and very geometric ideas of Fano's in \cite {Fa}, although Fano's original arguments are incomplete and needed serious refinements.   \medskip

{\bf Acknowledgements:} The author is a member of GNSAGA of the Istituto Nazionale di Alta Matematica ``F. Severi''. The author wants to thank Ivan Cheltsov and Yuri Prokhorov for  useful e--mail exchanges of ideas on the topic of this paper.

\section{Preliminaries} \label{sec:prel}

\subsection{Fano threefolds} In this paper we will consider \emph{Fano threefolds} $X$, i.e., $X$  is an irreducible, Gorenstein variety of dimension three, with at most canonical singularities, with ample  anticanonical (Cartier) divisor $-K_X$ such that the linear system $|-K_X|$ is base point free. 
We  set
$$
p=\frac {-K_X^3}2+1
$$
that is an integer, called the \emph{genus} of $X$. Then $\dim(|-K_X|)=p+1$ and the general surface in $|-K_X|$ is a K3 surface with at worst Du Val singularities. We will mainly focus on the case in which $-K_X$ is very ample, so that the morphism associated to the linear system $|-K_X|$ linearly normally embeds $X$ as a non--degenerate variety into $\p^{p+1}$. Then the general curve section of $X$ is a canonical curve of genus $p$ and  $\deg(X)=2p-2$. In this situation $X$ is arithmetically Gorenstein. The hypothesis that $X$ has at most canonical singularities implies that $X$ is not a cone.

The following proposition will be useful later:

\begin{proposition}\label{prop:flik} Let $X\subset \p^r$ be an irreducible, linearly normal, projective threefold such that its general hyperplane section is a K3 surface with at worst Du Val singularities and its general curve section is canonically embedded. Then $X$ is normal, Gorenstein, $\mathcal O_X(K_X)\cong \mathcal O_X(-1)$ and either $X$ has only rational singularities, in which case it is an anticanonically embedded Fano threefold, or it is a cone.  
\end{proposition}

\begin{proof} By \cite [Prop. 1.2]{CM}, $X$ is normal, Gorenstein, and $\mathcal O_X(K_X)\cong \mathcal O_X(-1)$. If $X$ has non--rational singularities, that $X$ is a cone by \cite {Is}. If $X$ has rational singularities, then $X$ has canonical singularities (see \cite [Introduction]{P01}) and it is a Fano threefold. \end{proof}

\subsection{Hyperelliptic Fano threefolds}

Let $X$ be a Fano threefold as above but $-K_X$  is not very ample. In this case $X$ will be called \emph{hyperelliptic} because the curve $C$, intersection of two general elements of $|-K_X|$,  is hyperelliptic. The hyperelliptic Fano threefolds have been classified in \cite [Thm. 1.5] {PCS}.  Typical examples of such threefolds are  double covers of a rational normal scroll threefold $V$ of degree $p-1$ in $\p^{p+1}$, branched along a (Weil) divisor of a linear system of the form
$4H-2(p-3)L$, where $H$ is the hyperplane class of $V$ and $L$ is a plane generator  of the scroll. 
We will need the following result which follows from \cite [Thm. 1.5, Prop. 1.10] {PCS}:

\begin{theorem}\label{thm:lops} Let $X$ be a hyperelliptic Fano threefold of genus $p\geq 10$. Then $X$ is rational. 
\end{theorem}

\subsection{The ideal of a Fano threefold} The following result is well known and it is a  consequence of the classical Enriques--Petri theorem for canonical curves (see, e.g.,  \cite [Thm. 2.14] {PCS}):

\begin{proposition} \label{prop:EP} Let $X\subset \p^{p+1}$ be an anticanonically embedded Fano threefold of genus $p$. Then either the ideal of $X$ is generated by quadrics or one of the following happens:\\ \begin{inparaenum}
\item [(i)] $p=3$ and $X$ is a quartic fourfold in $\p^4$;\\
\item [(ii)] $p=4$ and $X$ is a complete intersection of a quadric and a cubic in $\p^5$;\\
\item [(iii)] $p\geq 5$, any smooth curve section of $X$ is trigonal, the ideal of $X$ is generated by quadrics and cubics, the intersection of all quadrics containing $X$ is a 4--dimensional variety $V$ of minimal degree $p-2$ that is swept out by a 1--dimensional rational family $\mathcal F$ of 3--dimensional linear spaces and $X$ is cut out on $V$ by a cubic hypersurface containing  $p-4$ linear spaces of the family $\mathcal F$;\\
\item [(iv)] $p=6$, any smooth curve section of $X$ is the canonical image of a smooth plane quintic, the ideal of $X$ is generated by quadrics and cubics, the intersection of all quadrics containing $X$ is a 4--dimensional variety $V$ of minimal degree $4$ that is a cone with vertex a line $\ell$ over a Veronese surface of degree 4 in $\p^5$, and $X$ is cut out on $V$ by a cubic hypersurface containing a rank $3$ quadric on $V$ that is the cone over a conic of the Veronese surface with vertex $\ell$. 
\end{inparaenum}
\end{proposition}

The Fano threefolds of type (ii) and (iii) in the above proposition are said to be \emph{trigonal}. The Fano threefolds of type (iv) are said to present the \emph{plane quintic case}.
Note that, as soon as $p\geq 5$,  a trigonal Fano threefold is swept out by a pencil $\mathcal P$ of cubic surfaces contained in the 3--dimensional linear spaces of the family $\mathcal F$. 

In the paper \cite {PCS} there is a full classification of trigonal  Fano threefolds: there are  67 types of them.

\subsection{Reducible hyperplane sections of a Fano threefold} 

\begin{proposition}\label{prop:reduc}
Let $X\subset \p^{r}$ be an anticanonically embedded  Fano threefold. Let $S_0$ be a hyperplane section of $X$ that is reducible as $S_0=A+B$, with $A,B$ irreducible, distinct and reduced surfaces, with $A$ rational. Then $B$ is also rational. Moreover, if $C$ is  the intersection curve $C$ of $A$ and $B$, then its strict transform on the minimal desingularizations of $A$ and $B$ is anticanonical.  
\end{proposition}

\begin{proof} We can consider a family $f: \mathcal S\longrightarrow \mathbb D$, where $\mathbb D$ is a disc, the fibre of $f$ over $0$ is $S_0$ and the general fibre of $f$ is a general hyperplane section of $X$, that is a K3 surface $S$. Let $\phi: \mathcal X\longrightarrow \mathbb D$ be the semistable reduction of $f: \mathcal S\longrightarrow \mathbb D$. Clearly $\phi: \mathcal X\longrightarrow \mathbb D$ must be a type II degeneration (see \cite {K,PP}), whence the assertion immediately follows. \end{proof}

\subsection{Lines on Fano threefolds} \label{ssec:lines}

In this section we first prove the following result due to Fano (see \cite[\S2]{Fa}):

\begin{proposition}[Fano's Lemma]	\label{prop:lines} Let $X\subset \p^{p+1}$ be an anticanonically embedded Fano threefold and let $\mathcal R$ be any irreducible family of lines contained in $X$. Then either $\mathcal R$ has dimension 2, in which case the lines in $\mathcal R$ fill up a plane, or $\mathcal R$ has dimension at most 1.
\end{proposition}

\begin{proof} First of all we notice that the dimension of $\mathcal R$ is at most 2. Indeed, the dimension of $\mathcal R$ can be at most 4 and if it is 4 then $X$ is $\p^3$  (see \cite {BS}), a contradiction. If the dimension of $\mathcal R$ is 3, then $X$ can either be a quadric or a scroll in planes (see again  \cite {BS}), and both cases are not possible. Let us now assume that the dimension of $\mathcal R$ is 2. If the closure of the union of the lines of $\mathcal R$ is a surface, then this surface, containing a two--dimensional family of lines, is a plane. So we may assume that the closure of the union of the lines of $\mathcal R$ is not a surface, so that it is the whole $X$. 

Suppose first that, if $P\in X$ is a general point, then there is a unique line of $\mathcal R$ passing through $P$. Let $S$ and $S'$ be two general hyperplane sections of $X$ intersecting along a smooth canonical curve $C$. There is an obvious birational map $\pi: S\dasharrow S'$, that maps a general point $P\in S$ to the point $P'\in S'$, such that the line $\langle P, P'\rangle$ belongs to $\mathcal R$. Since $S,S'$ are minimal K3 surfaces, $\pi$ is actually a morphism. Moreover $C$ is pointwise fixed by $\pi$. This implies that $\pi$ is in fact induced by a projectivity between the hyperplanes spanned by $S$ and $S'$, that by abuse of notation we still denote by $\pi:  \langle S\rangle \longrightarrow \langle S'\rangle$. In fact this projectivity is a perspective, since $\pi$ fixes $C$ and therefore it fixes the subpace $\langle C\rangle=\langle S\rangle \cap \langle S'\rangle$ spanned by $C$. This immediately implies that $X$ is a cone, a contradiction.

Suppose next that   if $P\in X$ is a general point, there is more than one line of $\mathcal R$ passing through $P$. Let us fix  a general line $\ell$ in $\mathcal R$, two general points $P,P'\in \ell$ and the two tangent spaces $T_{X,P},T_{X,P'}$ to $X$ at $P,P'$ respectively. Then the join $\Pi:=\langle T_{X,P}, T_{X,P'}\rangle$ has dimension at most 5. We assume $\dim(\Pi)=5$, otherwise the proof runs in a similar way  (and the details can be left to the reader). 

\begin{claim}\label{cl:3} The linear space $\Pi$ contains the tangent space $T_{X,Q}$ for all points $Q$ of $\ell$ that are smooth for $X$. \end{claim}

\begin{proof} [Proof of the Claim] The family $\mathcal R$ can be interpreted as a surface $\Sigma$ on the Grassmannian $\mathbb G (1,r)$ of lines in $\p^r$, and $\ell$ as a general point of $\Sigma$. Take any tangent vector $\bf t$ to $\Sigma$ at $\ell$. As well know $\bf t$ can be interpreted as a projective map
$$
\tau: \ell \longrightarrow \p_\ell
$$
where $\p_\ell$ is the projective space of dimension $r-2$ whose points are the planes passing through $\ell$. The projective space $\p_\ell$  contains the linear subspace $\p_{\ell,\Pi}$ of dimension 3 whose points are the planes passing through $\ell$ and contained in $\Pi$. Since $\Pi$ contains $T_{X,P},T_{X,P'}$, it is immediate that $\tau(P)$ and $\tau(P')$ both belong to $\p_{\ell,\Pi}$, and therefore $\tau$ is a projectivity
$$
\tau: \ell \longrightarrow\p_{\ell,\Pi}.
$$
This immediately implies the assertion.\end{proof}

Let now $R_\ell$ be the scroll described by all lines in $\mathcal R$ that intersect $\ell$. 
By Claim \ref {cl:3}, $R_\ell$ is contained in $\Pi$. We may assume $R_\ell$ is irreducible, otherwise we substitute $R_\ell$ with an irreducible component of it. Now we have two possibilities:\\ \begin{inparaenum}
\item [(i)] if $T$ is a general point of $R_\ell$, all lines of $\mathcal R$ passing through $T$ are contained in $R_\ell$;\\
\item [(ii)] we are not in case (i), and then as $T$ varies in $R_\ell$, the lines of $\mathcal R$ passing through $T$ fill up the whole of $\mathcal R$.
\end{inparaenum}

\begin{claim} Case (i) above cannot happen. 	\end{claim}

\begin{proof}[Proof of the Claim] Suppose we are in case (i). Then either $R_\ell$ is a plane or $R_\ell$ would be a scroll surface  such that through its general point there is more than one line passing, and this can happen only if $R_\ell$ is a quadric. But then $X$ would be swept out by a 1--dimensional family of planes or of quadrics, and therefore its general hyperplane section would contain a 1--dimensional family or rational curves, a contradiction. 
\end{proof}

So we are left with case (ii). In this case let $T$ be a general point of $R_\ell$ and $\ell'$ a line of $\mathcal R$ passing through $T$ and not contained in $R_\ell$. Set $\Pi'=\langle \Pi, \ell'\rangle$ that has dimension at most 6. Then $\Pi'$ contains the ruled surface $R_{\ell'}$, because any line of $\mathcal R$ intersecting $\ell'$ intersects also $R_\ell$, hence it intersects $\Pi$. But then, since $\ell'$ is also a general line of $\mathcal R$, all lines in $\mathcal R$ intersect also $R_{\ell'}$, hence they must lie in $\Pi'$, thus $X$ is contained in $\Pi'$. We conclude that $p+1\leq 6$, hence $p\leq 5$. On the other hand the Fano threefolds of genus $p\leq 5$ are well known and they do not contain
a two dimensional family of lines. \end{proof} 

We will also need the following:

\begin{proposition}\label{prop:isk} Let $X\subset \p^r$ be an anticanonically embedded Fano threefold and let $\ell$ be a line contained in the smooth locus of $X$. Then for the normal bundle $N_{\ell,X}$ of $\ell$ in $X$ there are only the following possibilities
\begin{equation}\label{eq:line}
\text{either} \quad N_{\ell,X}\cong \mathcal O_\ell\oplus \mathcal O_\ell(-1)\quad \text{or}\quad N_{\ell,X}\cong \mathcal O_\ell(-2)\oplus \mathcal O_\ell(1).
\end{equation}
\end{proposition}

\begin{proof} This has been proved in \cite [Lemma (3.2)]{Isk} under the hypothesis that $X$ is smooth. The proof runs exactly in the same way in our case. \end{proof}

If for a line $\ell$ contained in the smooth locus of $X$ the first of the two alternatives in \eqref {eq:line} occurs, one says that $\ell$ is an \emph{ordinary line} of $X$.

\section{Rationality of trigonal Fano threefolds} \label{sec:trig}

As already said, in the paper \cite {PCS}  there is a full classification of trigonal  Fano threefolds, and moreover in that paper the authors, going through the classification, also determine which of these Fano threefolds are rational or not.  In this section, following ideas of Fano in \cite {Fa}, we  will   prove the following result:

\begin{theorem}\label{thm:trig} Let $X\subset \p^{p+1}$ be a  trigonal Fano threefold of genus $p\geq 8$. Then $X$ is rational. 
\end{theorem}

This theorem could be deduced from the results in \cite {PCS}. However the proof presented here is independent form the one in \cite {PCS}, is conceptually rather easy, very geometric,  and does not resort to the full classification of  \cite {PCS}. 

\begin{proof}[Proof of Theorem \ref {thm:trig}] Recall that $X$ is contained in a rational normal 4--dimensional scroll in 3--dimensional spaces $V$ of degree $p-2$. If we denote by $H$ the hyperplane class of $V$ and by $\Pi$ the class of a 3--dimensional space of the rational family $\mathcal F$, one has the linear equivalence of Weil divisors on $V$
$$
X\sim 3H-(p-4)\Pi.
$$

First we consider the case in which $V$ is a cone. Note that the vertex of the cone $V$ cannot have dimension $2$. In fact, if the vertex $P$ of $V$ is a plane, a cubic hypersurface cutting out $X$ on $V$, that has to contain $p-4$ spaces of the family $\mathcal F$, must contain $P$, and therefore the spaces of the family $\mathcal F$ would cut out on $X$, off $P$,  surfaces of degree at most 2, and this would imply that the general hyperplane section of $X$ is swept out by rational curves, a contradiction. Hence the vertex of $V$ is either a point or a line. 

Suppose first the vertex $O$ of $V$ is a point. By the same argument we made before, a cubic hypersurface $F$ cutting out $X$ on $V$ must contain $O$, because it has to contain $p-4$ spaces of the family $\mathcal F$. We will prove that $F$ must be singular at $O$. This immediately implies that $X$ is rational, because then the cubic surfaces cut out by the spaces in $\mathcal F$ on $X$ are singular at the fixed point $O$. Then, by projecting $X$ from $O$ to a hyperplane, we have a birational map of $X$ to the hyperplane section of $V$ that is rational. So we dispose of this case by proving the:

\begin{claim}\label{cl:uno} In the above setting $F$ is singular at $O$.
\end{claim}

\begin{proof}[Proof of the Claim] Suppose, by contradiction, that $F$ is smooth at $O$. We may suppose that $F$ contains $p-4$ distinct spaces $\Pi_i$, $1\leq i\leq p-4$, of the family $\mathcal F$. Then $\Pi_1,\ldots, \Pi_{p-4}$ have to be contained in the tangent hyperplane to $F$ at $O$. This hyperplane cuts out on $V$, off $\Pi_1,\ldots, \Pi_{p-4}$,  a residual 3--dimensional variety $Q$ of degree 2. If $Q$ is irreducible, then it is a quadric that intersects the linear spaces of $\mathcal F$ in planes of a ruling, i.e.,  $Q$ is a quadric of rank 4 in $\p^4$, with vertex $O$. So $Q$ is cut out by the cubic $F$ along $p-4\geq 4$ planes of a ruling, which implies that $F$ has to contain $Q$, a contradiction, because this would imply that $Q$ is a component of $X$. If $Q$ is reducible, then it splits in two 3--dimensional spaces, one of which $\mathfrak P$ intersects the linear spaces of $\mathcal F$ in planes. Then the same argument as above proves that $F$ should contain $\mathfrak P$, a contradiction. \end{proof}

Next, we assume that the vertex of $V$ is a line $\ell$. Again a cubic hypersurface $F$ cutting out $X$ on $V$ must contain $\ell$. As in the previous case the rationality of $X$ is a consequence of the following:

\begin{claim} The cubic form $F$ is singular in some point on $\ell$.
\end{claim}

\begin{proof}[Proof of the Claim] Suppose, by contradiction, that $F$ is smooth all along $\ell$. Then a direct computation, that can be left to the reader, shows that the tangent hyperplane to $F$ at a point $P\in \ell$ varies with $P$.  Again we may suppose that $F$ contains $p-4$ distinct spaces $\Pi_i$, $1\leq i\leq p-4$, of the family $\mathcal F$. Since $\Pi_1,\ldots, \Pi_{p-4}$ have to be contained in the tangent hyperplane $\Pi$ to $F$ at every point $P\in \ell$, then the span of $\Pi_1,\ldots, \Pi_{p-4}$ has dimension not larger than $p-1$. Hence we can find a pencil of hyperplanes of $\p^{p+1}$ containing $\Pi_1,\ldots, \Pi_{p-4}$. The hyperplanes of this pencil cut out on $V$, off $\Pi_1,\ldots, \Pi_{p-4}$,  a 1--dimensional family $\mathcal Q$ of residual 3--dimensional varieties  of degree 2. 

Suppose  the general variety $Q\in \mathcal Q$ is irreducible, so that it is a rank 3 quadric with vertex the line $\ell$. By projecting down from $\ell$, $V$ maps to a rational normal scroll surface $\Sigma\subset \p^{p-1}$ of degree $p-2$, that has a 1--dimensional family of irreducible conics, and this implies right away that $p-2=\deg(\Sigma)\leq 4$, a contradiction, because we are assuming $p\geq 8$. 

If the general variety $Q\in \mathcal Q$ is reducible, then $V$ has to contain a 3--dimensional linear space $\mathfrak P$ intersecting the linear spaces of $\mathcal F$ in planes. Then the same argument as in the proof of Claim \ref {cl:uno} shows that $p-4\leq 3$, which leads again to a contradiction. \end{proof}

So we can assume now that $V$ is smooth, hence 
$$V=\p(\O_{\p^1}(d_1)\oplus \O_{\p^1}(d_2)\oplus \O_{\p^1}(d_3)\oplus \O_{\p^1}(d_4))$$
with 
$$d_1\geq d_2\geq	 d_3\geq d_4>0$$
and 
$$ d_1+d_2+d_3+d_4=p-2.$$
We examine separately the three different cases:\\ \begin{inparaenum}
\item [(i)] $d_1=d_2=d_3=d_4=d$;\\
\item [(ii)] $d_1>d_2=d_3=d_4=d$;\\
\item [(iii)] all other cases, i.e.,  $d_1\geq d_2>d_3\geq d_4$.
\end{inparaenum}

\begin{claim} \label{cl:2} In case (i) the only possibility is $d=2$, hence $p=10$ and the resulting Fano threefolds are rational.
	\end{claim}
	
	\begin{proof}[Proof of the Claim] In this case $V$ is nothing else than $\p^1\times \p^3$ embedded via the linear system of divisors of bidegree $(d,1)$. Hence $V$ contains a 3--dimensional family $\mathcal C$ of rational normal curves of degree $d=\frac {p-2}4$ parametrized by $\p^3$, which are the images of the lines $\p^1\times \{x\}$, with $x\in \p^3$. 

Let $F$ be a cubic hypersurface that cuts out $X$ on $V$ off $p-4$ distinct spaces $\Pi_i$, $1\leq i\leq p-4$, of the family $\mathcal F$. Since $F$ cannot contain all the curves of $\mathcal C$, one must have
$$
p-4\leq 3d=3\frac {p-2}4
$$
which implies $p\leq 10$. But since $p-2\equiv 0$ modulo 4, the only possibility is $p=10$ and $d=2$. In this case it is immediate that $X$ is isomorphic to $\p^1\times S$, where $S$ is the cubic surface cut out by $F$ on a 3--dimensional space of the family $\mathcal F$, proving the assertion.
\end{proof}

\begin{claim} In case (ii) there are no Fano threefolds with $p\geq 8$.
	\end{claim}
	
\begin{proof}[Proof of the Claim] As in the proof of Claim \ref {cl:2}, we must have
$$
p-4\leq 3d\leq 3\frac {p-3}4
$$
i.e., $p\leq 7$, a contradiction.\end{proof}

To examine case (iii), and finish the proof of the theorem, we do the following. Let $\Pi$ be a general 3--dimensional linear space of the family $\mathcal F$ of $V$. Let us project down $X$ and $V$ from $\Pi$ to a $\p^{p-3}$. In this projection $V$ maps birationally to $V'$ and $X$ maps to $X'$. The variety $V'$ is the rational normal scroll of degree $p-6$ that is the image of  
$$\p(\O_{\p^1}(d_1-1)\oplus \O_{\p^1}(d_2-1)\oplus \O_{\p^1}(d_3-1)\oplus \O_{\p^1}(d_4-1))$$
in $\p^{p-3}$ via its $\O(1)$ line bundle. 

We have:\\ \begin{inparaenum}
\item [(a)]  $V'$ is smooth as soon as $d_4\geq 2$ and in this case the projection of $V$ to $V'$ induces an isomorphism $V\longrightarrow V'$ that restricts to an isomorphism $X\longrightarrow X'$;\\
\item [(b)] $V'$ is a cone with vertex a single point $O$ if $d_3>d_4=1$;\\
\item [(c)] $V'$  is a cone with vertex a line $\ell$,  if $d_3=d_4=1$.
\end{inparaenum}

 In case (b) the projection of $V$ to $V'$ induces a morphism $V\longrightarrow V'$, that contracts a line $r$  to  $O$ and it is an isomorphism between $V\setminus r$ and $V'\setminus \{O\}$. In   case  (c) the projection of $V$ to $V'$ induces a morphism $V\longrightarrow V'$, that contracts a smooth quadric $Q$  to  $\ell$ and it is an isomorphism between $V\setminus Q$ and $V'\setminus \ell$. 
 
 In either case  the projection of $V$ to $V'$ induces a birational morphism $X\longrightarrow X'$. This implies that, if $A$ is a general cubic surface in the pencil $\mathcal P$ cut out on $X$ by the  linear subspaces of the family $\mathcal F$, then the general hyperplane section of $X$ containing $A$ is of the form $A+B$, with $B$ also irreducible. By Proposition \ref {prop:reduc}, $B$ is also rational, and it is birationally mapped to the general hyperplane section of $X'$, whose general hyperplane section is thus rational. Then by the results in \cite {CF}, $X'\subset \p^{p-3}$, with $p-3\geq 5$, is rational, hence $X$ is rational, finishing our proof. 
\end{proof}

\begin{remark}\label{rem:rat} Theorem \ref {thm:trig} is sharp. Indeed, in \cite {PCS} there are examples of non--rational Fano threefolds of genus $p\leq 7$ that are not cut out by quadrics. 
\end{remark} 

\begin{remark}\label{rem:rat1}It is possible that the same ideas as in the proof of Theorem \ref {thm:trig} can be used to prove Theorem \ref {thm:lops}, but we do not dwell on this here.
\end{remark}

\section{Fano threefolds containing a plane}\label{sec:plane}

In this section we will consider the rationality question for Fano threefolds containing a plane. In view of Theorem \ref {thm:trig} we can focus on the case in which the Fano threefold is cut out by quadrics. First we prove the following:

\begin{lemma}\label{lem:propl} Let $X\subset \p^{p+1}$, with $p\geq 5$, be an anticanonically embedded Fano threefold of genus $p$ that is non--trigonal and does not present the plane quintic case and contains a plane $\Pi$. Then the projection from $\Pi$ induces a birational map $\pi: X\dasharrow X'\subset \p^{p-2}$.
\end{lemma}

\begin{proof} Let $P\in X$ be a general point. Consider the 3--dimensional linear space $\Pi_P=\langle P,\Pi\rangle$. Let us look at the intersection of $\Pi_P$ with $X$. Since $X$ is cut out by quadrics (see Proposition \ref {prop:EP}), the intersection of $X$ with $\Pi_P$ consists of $\Pi$ plus a linear subspace $L$ of $\Pi_P$ containing $P$. The subspace $L$ cannot be a plane, since otherwise $X$ would be a scroll in planes, its general hyperplane section would be a scroll, and this is not possible. Moreover $L$ cannot be a line, by Fano's Lemma \ref {prop:lines}. So $L=\{P\}$ and the assertion follows. \end{proof}

Next we can prove the:

\begin{theorem}\label{thm:propl} Let $X\subset \p^{p+1}$, be an anticanonically embedded  Fano threefold of genus $p\geq 7$ that is non--trigonal and contains a plane $\Pi$. Then $X$ is rational.
\end{theorem}

\begin{proof} Let $S$ be a general hyperplane section of $X$ containing $\Pi$. Then, by Lemma \ref 
{lem:propl}, we have $S=\Pi+F$, where $F$ is irreducible and reduced. By Proposition \ref
{prop:reduc}, $F$ is rational, and this implies that if $X'\subset \p^{p-2}$ is the image of the projection of $X$ from $\Pi$, then the general hyperplane section of $X'$ is rational. Since $p-2>4$, by the results in \cite {CF}, $X'$ is rational, and therefore, by Lemma \ref 
{lem:propl}, $X$ is rational.
\end{proof}

\begin{remark}\label{rem:three} Let $X\subset \p^{p+1}$ be an anticanonically embedded Fano threefold of genus $p\geq 5$ that is non--trigonal, does not present the plane quintic case, and contains a plane $\Pi$.
If $p=5$, then $X$ is rational by Lemma \ref 
{lem:propl}. If $p=6$ instead, it could be the case that the image $X'\subset \p^{4}$ of the projection of $X$ from $\Pi$ is a smooth cubic threefold, in which case $X$, that is birational to $X'$, is unirational but not rational. This can actually occur, as the following example (due to Fano, see \cite {Fa0}) shows.

  Let $X'\subset \p^4$ be a smooth cubic hypersurface and let $C\subset X'$ be a sufficiently general rational normal cubic curve contained in $X'$. Note that the span of $C$ is a hyperplane in $\p^4$, so that $C$ sits on a unique hyperplane section $Y$ of $X'$, that, by generality,  we may assume to be smooth.  
 The linear system $|\mathcal I_{C,\p^4}(2)|$ of quadrics passing through $C$ has dimension 7, and it determines a rational map $\phi: X'\dasharrow X\subset \p^7$, with $X$ a  threefold. It is immediate to see that $X$ is an anticanonically embedded Fano threefold of genus $6$. It contains a plane $\Pi$, that is the image of the cubic surface $Y$ via the map $\phi$. The projection of $X$ from $\Pi$ is exactly $X'$. The reader will check that $X$ has six double points (that are the images via $\phi$ of the six chords of $C$ lying on $Y$) all located on the plane $\Pi$. If $S$ denotes a hyperplane section of $X$, the linear system $|S-\Pi|$ of surfaces of degree 9 on $X$, cuts out on $\Pi$ the complete linear system of dimension 3 of cubics with six base points at the double points of $X$ on $\Pi$. This implies that there is a unique hyperplane that is tangent to $X$ all along $\Pi$. 
  
However, there are also  anticanonically embedded Fano threefolds of genus 6 containing a plane that are rational, as the following example (again due to Fano, see \cite {Fa0}) shows.

Let $X'\subset \p^4$ be a smooth quadric hypersurface. Let $Y$ be a smooth complete intersection of $X'$ with another quadric, so that $Y$ is a Del Pezzo surface of degree 4 in $\p^4$, that is the image of $\p^2$ via the linear system $(3;1^5)$ of cubic curves passing through 5  points in general position. We can consider a smooth curve $C$ of degree 9 and genus 6 on $Y$. This is the image of a general curve in the linear system $(6;2^4,1)$ of sextics with 4  double base points and one simple base points at the base points of the system $(3;1^5)$. The linear system $\mathcal L:=|\mathcal I_{C,X'}(3)|$ cut out on $X'$ by the cubic hypersurfaces passing through $C$ has dimension 7, and it determines a rational map $\psi: X'\dasharrow X\subset \p^7$, with $X$ an anticanonically embedded Fano threefold of genus $6$. It contains a plane $\Pi$, that is the image of the quartic surface $Y$ via the map $\psi$. The projection of $X$ from $\Pi$ is exactly $X'$. One checks that $X$ has five double points (that are the images via $\psi$ of the five three--secant lines of $C$ lying on $Y$) all on $\Pi$. If $S$ denotes a hyperplane section of $X$, the linear system $|S-\Pi|$ cuts out on $\Pi$ the complete linear system of dimension 4 of cubics with five base points at the double points of $X$ on $\Pi$. \end{remark}

\begin{remark}\label{rem:two} As we saw in Remark \ref {rem:three},  Theorem \ref {thm:propl}  is sharp, since there are Fano threefolds of genus $6$ cut out by quadrics, containing a plane and not rational, whereas, as we saw, Fano threefolds of genus 5, complete intersection of three quadrics, containing a plane are rational. More precisely (see \cite {Fa0}), consider a smooth cubic surface $Y$ in $\p^3$ that is the image of $\p^2$ via the linear system $(3;1^6)$ of cubic curves passing through 6  points in general position.  Consider  a smooth curve $C$ of degree 9 and genus 9 on $Y$.  
This is the image of a general curve in the linear system $(7;2^6)$ of septics with  double base points  at the base points of the system $(3;1^6)$. Consider the linear system $|\mathcal I_{C,\p^3}(4)|$, that has dimension 6. It determines a rational map $\eta: \p^3\dasharrow X\subset \p^6$, where $X$ is an anticanonically embedded Fano threefold of genus $5$ containing a plane $\Pi$ that is the image of the cubic $Y$. This $X$ is the most general Fano threefold of genus $5$ containing a plane $\Pi$. It has exactly six double points along the plane $\Pi$, corresponding to the contraction of the six 4--secant lines to $C$ contained in the surface $Y$.  

It is interesting to look at the cases $3\leq p\leq 4$, in which an anticanonically embedded Fano threefold is never cut out by quadrics.

The case $p=4$ is well known. Let $X\subset \p^5$ be a general  Fano threefold of genus 4 containing a plane $\Pi$. Remember that $X$ is a complete intersection of type $(2,3)$. The projection from $\Pi$ to a plane endowes $X$ with a birational structure of conic bundle. Actually (see \cite[Ex. 4.10.3]{B}), $X$ is birational to a \emph{standard conic bundle} (see \cite [Def. 3.1]{P}) on $\p^2$ with discriminant curve of degree 7, and therefore $X$ is not rational (see \cite [Thm. 4.9]{B}). 

The case $p=3$ of a quartic threefold $X$ in $\p^4$ containing a plane $\Pi$ is still rather mysterious. A general such quartic $X$ has exactly 9 ordinary double points along the plane $\Pi$ and no other singularity. Such an $X$ is unirational, but it is conjectured to be in general not rational (see \cite [Conj. 1.1]{G}). However there are special such quartics that are rational. Consider in fact the following example. 

Look again at a general anticanonically embedded Fano threefold $X$ of genus $5$ containing a plane $\Pi$. There are lines of $X$ not intersecting $\Pi$. If we think to $X$ as the image of the map $\eta: \p^3\dasharrow X\subset \p^6$ we considered above, we get such lines as images of trisecant lines to the curve $C$ of degree 9 and genus 9 on $Y$. If $\ell$ is such a line, project $X$ down to $\p^4$ from $\ell$. One checks that the image is a quartic hypersurface $X'$ containing a plane $\Pi'$ which is the image of $\Pi$, and it is rational as well as $X$. Note that, besides the plane $\Pi'$, $X'$ contains also a scroll, that is the image of the exceptional divisor of the blow--up of $X$ along $\ell$. 

Another interesting example is due to Fano  (see again \cite {Fa0}). Consider in $\p^6$ a smooth rational normal scroll $R$ of degree 3  in a codimension 2 linear subspace $T$. Then consider an anticanonically embedded Fano 3-fold $X$ of genus 5 containing $R$, obtained as the complete intersection of three general quadrics in $\p^6$ containing $R$. The pencil of hyperplanes through $T$ cuts out on $X$, off $R$, a pencil $\mathcal P$ of surfaces of degree 5, that are rational (see Proposition \ref {prop:reduc}), and are, in general, smooth Del Pezzo surfaces. The pencil $\mathcal P$ cuts out on $R$ a pencil of anticanonical curves (see again Proposition \ref {prop:reduc}), that are curves  of genus 1 and degree 5, with 8 base points that are double points for $X$. Now let us project $X$ down to $\p^4$ from a general line generator $\ell$ of $R$. The image of the projection is a quartic hypersurface $X'$, containing a plane $\Pi$ that is the image of $R$ under the projection. Along $\Pi$, the hypersurface $X'$ has, as expected, 9 double points, i.e., the images of the 8 double points of $X$ along $R$ and one further double point arising from the contraction of the line directrix of $R$. The pencil $\mathcal P$ is mapped via the projection to the pencil $\mathcal P'$ of cubic surfaces cut out on $X'$ off $\Pi$ by the hyperplanes through $\Pi$. Each quintic surface $\Phi\in \mathcal P$ intersects $\ell$ in two points, because an anticanonical curve of $R$ intersects $\ell$ in two points. In the projections these two points produce a pair of skew lines $r,r'$ on the image $\Phi'$ of $\Phi$. Since, as it is well known, $\Phi'$ is rational in the field of rationality of $r$ and $r'$, we conclude that $X'$ is rational, hence $X$ is rational. Of course $X'$ contains also  a scroll, that is the image of the exceptional divisor of the blow--up of $X$ along $\ell$. 
\end{remark}

\section{Two applications}\label{sec:appl}

In this section, inspired by arguments of Fano in \cite {Fa}, we will give two applications of Theorem \ref {thm:propl}  to rationality of non--hyperelliptic and non--trigonal  Fano threefolds of large enough genus. 
The first one is the following:

\begin{theorem}\label{thm:ratbig} Let $X\subset \p^{p+1}$ be an anticanonically embedded, non--trigonal Fano threefold of genus $p\geq 17$. Then $X$ is rational.
\end{theorem}

\begin{proof} Let $P\in X$ be a general point. What we will do is to consider the \emph{tangential projection} at $P$, i.e., the projection of $X$ from the tangent space $T_{X,P}$ to $X$ at $P$. To do this we need a few preliminary facts. First of all we prove the:

\begin{claim}\label{cl:a} In the above set up, $T_{X,P}$ intersects $X$ only at $P$.
\end{claim}

\begin{proof}[Proof of the Claim] Suppose, by contradiction that $T_{X,P}$ intersects $X$ at some point $P'\neq P$. Then the line $\langle P,P'\rangle$ has intersection multiplicity 2 with $X$ at $P$ and contains also $P'$, so it sits on $X$, because $X$ is cut out by quadrics (see Proposition \ref {prop:EP}). This contradicts Fano's Lemma \ref {prop:lines}. 
\end{proof}

Consider next the \emph{second fundamental form} $II_P$ of $X$ at $P$ (see \cite [p. 263 and ff.]{GH}). Remember that $II_P$ can be considered as a linear system of conics on the exceptional divisor $E\cong \p^2$ of the blow--up of $X$ at $P$. This linear system has no base points. In fact the base points of the second fundamental form correspond to \emph{asymptotic directions} to $X$ at $P$, i.e., the directions of lines having with $X$ at $P$ intersection multiplicity at least 3 (see \cite[p. 63]{Ru}). There is no such a direction by Fano's Lemma \ref {prop:lines}. All this  implies the following:

\begin{claim}\label{cl:b} Consider the linear system $\mathcal H_P$ of hyperplane sections of $X$ that are tangent to $X$ at $P$, i.e., sections made with hyperplanes that contain $T_{X,P}$. Then the general member of $\mathcal H_P$ has a double point of type $A_1$ at $P$ and no other singularity except the intersections with the singular locus of $X$. Moreover three general members of $\mathcal H_P$ have intersection multiplicity 8 at $P$. Finally the strict transform of $\mathcal H_P$ on the blow--up of $X$ at $P$ has no base points and its general member is a K3 surface.
\end{claim}

Next we make the:

\begin{claim}\label{cl:c} The Fano threefold $X$ is not defective, i.e., the variety of secant lines to $X$ has the expected dimension 7. 
\end{claim}

	\begin{proof}[Proof of the Claim] Defective threefolds have been classified (see \cite [Thm. 1.1] {CC}). An irreducible, non degenerate, projective  threefold, not a cone,  in $\p^r$ is defective if and only if $r\geq 6$ and it is of one of the following types:\\ \begin{inparaenum}
\item [(1)] it sits in a 4--dimensional cone over a curve;\\
\item [(2)] $r=7$ and it sits in a 4--dimensional cone over a Veronese surface in $\p^5$;\\
\item [(3)] it is the 2--Veronese image of $\p^3$ in $\p^9$ or a projection of it in $\p^8$ or $\p^7$;\\
\item [(4)] $r=7$ and it is the hyperplane section of the Segre embedding of $\p^2\times \p^2$ in $\p^8$.\end{inparaenum}	

Since $p\geq 17$, our $X$ cannot be of types $(2),(3),(4)$. Suppose it is of type $(1)$. Then it sits in a 4--dimensional cone over a curve, that is swept out by a 1--dimensional family $\mathcal V$ of linear spaces of dimension 3. Since $X$ is cut out by quadrics,  the intersection of $X$ with the 
	spaces of $\mathcal V$ are surfaces of degree at most 1. Then the general  hyperplane sections of $X$ would be swept out by rational curves, a contradiction.  \end{proof}
	
	It is a consequence of the classical Terracini's Lemma that, if $X$ is not defective, then the  tangential projection of $X$ at a general point $P\in X$ is generically finite to its image (see \cite [Thm. 1.4.1, Prop. 1.4.10]{Ru}). We denote by $\tau: X\dasharrow X'\subset \p^{p-3}$ the tangential projection in question and we let $\nu$ be its degree. 

 \begin{claim}\label {cl:d} One has $\nu\leq 2$. If $\nu=2$, then $X'$ is a threefold of minimal degree $d-5$ in $\p^{p-3}$.
 \end{claim}
 
 \begin{proof}[Proof of the Claim] Since $\deg(X)=2p-2$, by Claim \ref {cl:b} we have that the degree of $X'$ is 
 $$
 \frac {2p-2-8}\nu\geq p-5
 $$
 which  implies that $\nu\leq 2$ and if $\nu=2$ the equality holds, proving the assertion.
 \end{proof}
 
 We first examine the case $\nu=2$. In this case the blow--up $\widetilde X$ of $X$ at $E$ turns out to be a hyperelliptic Fano threefold, of genus $p-4$. Under our hypotheses, we have $p-4\geq 10$, hence $\widetilde X$ is rational by Theorem \ref {thm:lops}, thus $X$ is rational, as wanted.
 
 Assume next $\nu=1$, i.e., the tangential projection is birational. 
 
 \begin{claim}\label{cl:fano} In this setting, $X'$ is an anticanonically embedded Fano threefold of genus $p-4$.
 \end{claim}
 
 \begin{proof}[Proof of the Claim \ref {cl:fano}] The general hyperplane section $S'$ of $X'$ is 
the birational image under the tangential projection $\tau$ of  a general tangent hyperplane section $S$ of $X$ at $P$, and this is a K3 surface with at worst Du Val singularities. Precisely, $S'$ is isomorphic to the (partial) desingularization of $S$ at the $A_1$ singular point that $S$ has at $P$. By Proposition \ref {prop:flik}, $X'$ is normal and to prove the claim, it suffices to show that $X'$ is not a cone. We argue by contradiction and suppose that $X'$ is a cone, with vertex a point $O$. 

We first note that $O$ must be a fundamental point of the inverse of the tangential projection $\tau: X\dasharrow X'$ from $P\in X$. Otherwise there would be a point $Q\in X$ such that $\tau(Q)=O$ and $\tau$ would induce an isomorphism between the neighborhoods of $Q\in X$ and $O\in X'$. This is a contradiction, because it would imply that $Q$ is a non--rational singularity of $X$, which is impossible by Proposition \ref {prop:flik} because $X$ is a Fano threefold. 

Since $O$ is a fundamental point of the inverse of the tangential projection, its inverse image $\gamma$ via $\tau$ is positive dimensional and connected, by Zariski's Main Theorem, and it is contained in the 4--dimensional linear space spanned by $T_{X,P}$ and any point of $\gamma$. By Claim \ref {cl:a}, $\gamma$ can only be a curve that intersects $T_{X,P}$ only at $P$. 

The partial desingularization $\widetilde X'\longrightarrow X'$ obtained by blowing up at $O$, has exceptional divisor over $O$ isomorphic to the general hyperplane section of $X'$, that is a K3 surface. On the other hand it is possible to obtain such a partial desingularization also by blowing up $X$ along $\gamma$. This gives a contradiction because in this way we would never get an exceptional divisor over $O$ birational to a K3 surface. \end{proof}

The tangential projection has an indeterminacy point at $P$, that we can resolve by blowing up $P$. So it becomes a morphism $\tau: \widetilde X\longrightarrow X'$. Then $X'$ contains the image $\Sigma$ of the exceptional divisor $E$ of $\widetilde X$ over $P$, via the second fundamental form $II_P$. Since $II_P$ has no base points, we have $5\geq \dim (II_P)\geq 2$, and we examine separately the various cases.
 
 The case $\dim (II_P)= 2$ does not happen. Indeed in this case $\Sigma$ would be a plane of quadruple points for $X'$, which is not possible, since the general hyperplane section of $X'$ is a K3 surface with at most Du Val singularities.
 
The case $\dim (II_P)= 3$ does not happen either. In fact in this case $\Sigma$ would be a \emph{Roman Steiner surface} of degree 4 in a 3--dimensional linear space, i.e., the projection in $\p^3$ of a Veronese surface in $\p^5$. Since $X'$ is cut out by quadrics and cubics (at most), this cannot be the case. 
  
Suppose $\dim (II_P)= 4$. Then $\Sigma$ is a projection  of a Veronese surface in a 4--dimensional linear space. Such a surface is not contained in any quadric. So any quadric containing $X'$ has to contain the whole 4--dimensional linear space containing $\Sigma$.  This implies  that $X'$ is trigonal. Since in our case $p-4\geq 8$, then, by Theorem \ref {thm:trig}, $X'$ is rational, and therefore also $X$ is rational.
   
Suppose finally that $\dim (II_P)= 5$. Then $\Sigma$ is a Veronese surface in a 5--dimensional linear space. Note that, by what we have seen  above, a general hyperplane section $S'$ of $X'$ (that is a K3 surface) intersects $\Sigma$ along an irreducible rational normal quartic curve $D$ along which $S'$ is smooth: $D$ is the $(-2)$--curve that is the resolution of the $A_1$ double point that the general tangent hyperplane section of $X$ at $P$ has at $P$. Accordingly, if there are singular points of $X'$ on $\Sigma$, they do not fill up a curve (because there are no such singular points of $D$), so there can be at most finitely many singular points of $X'$ along $\Sigma$.

Let now $\Gamma$ be a general conic of $\Sigma$. By the above considerations, $X'$ is smooth along $\Gamma$. 
 If $X'$ is trigonal, then as above $X'$ is rational and so also $X$ is rational.  So we can assume $X'$ is not trigonal, hence it is cut out by quadrics. Then the plane $\Pi$  spanned by $\Gamma$ intersects $X'$ scheme theoretically only along $\Gamma$, otherwise it would be contained in $X'$, and then, by the genericity of $\Gamma$ in $\Sigma$, $X'$ would contain the whole secant variety of $\Sigma$ that has dimension $4$, a contradiction. 
 
 \begin{claim}\label {cl:e} The general hyperplane section of $X'$ containing $\Gamma$ is smooth along $\Gamma$  and it is a K3 surface.
 \end{claim}
 
 \begin{proof}[Proof of the Claim]   We imitate the proof of \cite [Lemma (3.2)(i)]{Isk}. Consider the linear system $\mathcal H_\Gamma$ of hyperplane sections of $X'$ containing $\Gamma$. By what we have seen above, the base locus scheme of $\mathcal H_\Gamma$ is just $\Gamma$. By Bertini's theorem, the general surface in $\mathcal H_\Gamma$ is irreducible and it has singular points only at the intersection with the singular locus of $X'$ and perhaps along $\Gamma$. However we will show that the general surface in $\mathcal H_\Gamma$ is smooth along $\Gamma$, and this will prove the claim. 
 
Consider the projection $\pi: X'\subset \p^{p-3} \dasharrow X''\subset \p^{p-6}$ of $X$ from the plane $\Pi$ spanned by $\Gamma$, that is given by the linear system $\mathcal H_\Gamma$. Let $\sigma: \p'\longrightarrow \p^{p-3}$ be the blow--up of $\Gamma$. Then the composition $\pi\circ \sigma$ is  a morphism $f: \p'\to \p^{p-6}$. If $\sigma': \widetilde X'\to X'$ is the blow--up of $X'$ along $\Gamma$, then $\widetilde X' \subset \p'$ and $\widetilde X'$ is smooth along the exceptional divisor $\mathcal E$ of the blow--up. We have the morphism $\tilde f$ that is the restriction of $f$ to $\widetilde X'$. If $H'$ denotes the general hyperplane section of $X'$, then $\tilde f$ is given by the linear system $|(\sigma')^*(H')-\mathcal E|$. This linear system has no base points, because it is the restriction to $\widetilde X'$ of a linear system on $\p'$ with the same properties, i.e., the linear system that determines the morphism $f$. The restriction of 
 $|(\sigma')^*(H')-\mathcal E|$ to $\mathcal E$ has also no base points, and it consists of sections of the map $\sigma': \mathcal E \longrightarrow \Gamma$. Hence the general surface $S''\in |(\sigma')^*(H')-\mathcal E|$ intersects $\mathcal E$ along a smooth, irreducible section $\Gamma'$ of  $\sigma': \mathcal E \longrightarrow \Gamma$. Then the map $\sigma: S''\to S':=\sigma(S'')$ is an isomorphism, and since $S'$ has no singular points along $\Gamma'$, then $S'$ is smooth along $\Gamma$, as wanted.  
 \end{proof}
 
 The projection $\pi: X'\dasharrow X''$ is generically finite. In fact, let $P'\in X'$ be a general point and let us consider the linear space of dimension 3 generated by $\Pi$ and $P'$. If the fibre of $P'$ via $\pi$ is positive dimensional, since it has to be cut out in $\langle \Pi,P'\rangle$ by quadrics passing through $\Gamma$, it could only be either a quadric or a plane or a line or a conic. The first two cases are impossible because then the general hyperplane section of $X'$ containing $\Gamma$ would be swept out by rational curves, a contradiction to Claim \ref {cl:e}. It is also impossible that the fibre is a line, because of Fano's Lemma \ref {prop:lines}. If the general fibre is a conic we also come to a contradiction because then the general hyperplane section of $X'$ containing $\Gamma$ would be swept out by conics, a contradiction again to Claim \ref {cl:e}. 
 
 Let $\widetilde X'$ be the blow--up of $X'$ along $\Gamma$ and let $\mathcal E$ be the exceptional divisor. Let $\mathcal H$ be the transform on $\widetilde X'$ of a general hyperplane section of $X'$. Then one computes $(\mathcal H-\mathcal E)^3=2p-16$. Let  $\mu$ be the degree of the map 
 $\pi: X'\dasharrow X''$. Then the degree of $X''$ is
 $$
 \frac {2p-16}\mu\geq p-8
 $$
 that, given the hypothesis on $p$, implies $\mu\leq 2$. Suppose $\mu=2$, that yields that $X''$ has degree $p-8$, i.e., $X''$ is a threefold of minimal degree and $\widetilde X'$ is a hyperelliptic Fano threefold of genus $p-7$. Since $p\geq 17$, $\widetilde X'$ is rational by Theorem \ref {thm:lops}, and therefore $X$ is rational, as wanted.
 
 Suppose next that $\mu=1$.
 
 \begin{claim}\label{cl:polp} If $\mu=1$, then $X''$ is an anticanonically embedded Fano threefold.
 \end{claim}
 
 \begin{proof}[Proof of Claim \ref {cl:polp}] By Claim \ref {cl:e}, the general hyperplane section  of $X''$ is a K3 surface with at worst Du Val singularities. By Proposition \ref {prop:flik}, $X''$ is normal and to prove the claim it suffices to show that $X''$ has only rational singularities. To prove this, let $O\in X''$ be any singular point. If the inverse of the projection $\pi: X'\dasharrow X''$ is defined at $O$, then $O$ is locally isomorphic to a point of $X'$, that is a rational singularity by Proposition \ref {prop:flik}, and we are done. Suppose that  the inverse of  $\pi: X'\dasharrow X''$ is not defined at $O$. Then the positive dimensional inverse image $Z$ of $O$ via $\pi$ sits in the 3--dimensional linear space spanned by the plane $\Pi$ and any point of $Z$. As we have seen above, $Z$ can only be either a quadric or a plane or a line or a conic. In any event, the singularity of $X''$ at $O$ is rational, as wanted. \end{proof}
 
 Now, if $X''$  is trigonal, it is rational because $p\geq 17$ (see Theorem \ref {thm:trig}). Suppose then $X''$ is non--trigonal. It contains the image of $\Sigma$ via the projection $\pi: X'\dasharrow X''$, and this image is a plane. Then $X''$ is rational because of Theorem \ref {thm:propl}. So $X$ is rational as desired. \end{proof}
 
 The second application is the following:
 
 \begin{theorem}\label{thm:line} Let $X\subset \p^{p+1}$ be an anticanonically embedded Fano threefold of genus $p\geq 12$ containing an ordinary line $\ell$ in its smooth locus. Then $X$ is rational.
 \end{theorem} 
 
 \begin{proof} Because of Theorem \ref {thm:trig} we may assume that $X$ is non--trigonal. Then the projection $\sigma: X\dasharrow X'\subset \p^{p-1}$ from the line $\ell$ is birational to the image $X'$. Indeed, let $P\in X$ be a general point and let $\Pi$ be the plane spanned by $\ell$ and $P$. The fibre of $\sigma$ through $P$ is cut out on $\Pi$, off $\ell$ by quadrics containing $\ell$. So this fibre can either be a line or the single point $P$. The former case cannot happen because of Fano's Lemma \ref {prop:lines}. This proves that $\sigma: X\dasharrow X'$ is birational. Similarly as in Claim  \ref {cl:e} (or arguing as in \cite [Lemma (3.2)(i)]{Isk}) one proves that the general hyperplane section of $X$ containing $\ell$ is smooth along $\ell$  and it is a K3 surface.
Then, arguing as in Claim \ref {cl:polp}, we see that  $X'$ is also an anticanonically embedded  Fano threefold of genus $p-2$. If it is trigonal, then it is rational by Theorem \ref {thm:trig} and $X$ is also rational. Suppose then it is not trigonal. Let $\widetilde X$ be the blow--up of $X$ along the line $\ell$ and let $R$ be the exceptional divisor. The rational map $\sigma$  extends to a morphism $\tilde \sigma: \widetilde X \longrightarrow X'$, and we can consider the image $R'$ of $R$ via $\tilde \sigma$. Since a general codimension 2 linear space passing through $\ell$ intersects $X$, off $\ell$, along a curve that intersects $\ell$
in three points, we see that $\deg (R')=3$ and, by Proposition \ref {prop:isk}, $R'\cong R$ is a smooth rational normal scroll of degree 3 in its 4--dimensional linear space. If $S$ is a general hyperplane section of $X$ containing $\ell$, its strict transform $\widetilde S\cong S$ on $\widetilde X$ intersects $R$ along a hyperplane section $C$ of $R$ in its 4--dimensional linear ambient space. Since $S$ is smooth along $\ell$, $\widetilde S$ is smooth along $C$ and $S'=\sigma(S)$ is smooth along its intersection with $R'$ that, in the isomorphism $R\cong R'$, is identified with the hyperplane section $C$ of $R'$.  This implies that, if there are singular points of $X'$ along $R'$, they are only finitely many. 

Now consider a general line  $\ell'$ of the ruling on $R'$.  By the above argument, $\ell'$ is contained in the smooth locus of $X'$. Since we are assuming $X'$ non-trigonal, as above the projection $\sigma': X'\dasharrow X''\subset \p^{p-3}$ is birational and, by the same argument as above,  $X''$ is a Fano threefold of genus $p-4$. If it is trigonal, then it is rational because of Theorem  \ref {thm:trig}, hence $X$ is rational. If it is not trigonal, it contains the image of $R'$ via the projection $\sigma'$, and this image is a plane. Then $X''$ is rational by Theorem \ref {thm:propl}, hence $X$ is rational again, as desired. \end{proof}

\end{document}